\documentclass[reqno,12pt]{amsart}

\usepackage{amssymb}
\usepackage[pdftex]{hyperref}

\hypersetup{colorlinks=true, linkcolor=black, citecolor=black}

\newcommand{\C}{{\mathbb C}}

\newcommand{\Z}{{\mathbb Z}}

\newcommand{\E}{{\mathsf E}}

\newcommand{\cA}{{\mathcal A}}
\newcommand{\cB}{{\mathcal B}}
\newcommand{\cH}{{\mathcal H}}
\newcommand{\cD}{{\mathcal D}}
\newcommand{\cP}{{\mathcal P}}
\newcommand{\cS}{{\mathcal S}}
\newcommand{\cT}{{\mathcal T}}

\newcommand{\alp}{\alpha}
\newcommand{\eps}{\varepsilon}
\renewcommand{\phi}{\varphi}

\theoremstyle{plain}
\newtheorem{lemma}{Lemma}
\newtheorem{proposition}{Proposition}
\newtheorem{corollary}{Corollary}
\newtheorem{theorem}{Theorem}

\newcommand{\refc}[1]{~\ref{c:#1}}
\newcommand{\refl}[1]{~\ref{l:#1}}
\newcommand{\refp}[1]{~\ref{p:#1}}
\newcommand{\refs}[1]{~\ref{s:#1}}
\newcommand{\reft}[1]{~\ref{t:#1}}
\newcommand{\refb}[1]{~\cite{b:#1}}
\newcommand{\refe}[1]{\eqref{e:#1}}

\newcommand{\seq}{\subseteq}
\newcommand{\sbs}{\subset}
\newcommand{\stm}{\setminus}
\newcommand{\est}{\varnothing}

\newcommand{\hC}{{\widehat{\C_p}}}

\newcommand{\hcA}{{\widehat{\cA}}}
\newcommand{\hcS}{{\widehat{\cS}}}

\author{Vsevolod F. Lev}
\email{seva@math.haifa.ac.il}
\address{Department of Mathematics, The University of Haifa at Oranim,
  Tivon 36006, Israel}

\author{Oriol Serra}
\email{oriol.serra@upc.edu}
\address{Department of Mathematics and Institute of Mathematics
   Barcelona-Tech, Universitat Polit\`ecnica de Catalunya}
\thanks{Supported by the Spanish Agencia Estatal de Investigación under
   projects PID2020-113082GBI00 and the Severo Ochoa and María de Maeztu
   Program for Centers and Units of Excellence in R\&D (CEX2020-001084-M)}

\title[Towards $3n-4$]%
  {Towards $3n-4$ \\ in groups of prime order}

\makeatletter
\@namedef{subjclassname@2020}{\textup{2020} Mathematics Subject Classification}
\makeatother


\begin{document}
\baselineskip = 16pt

\maketitle

\begin{abstract}
We show that if $A$ is a subset of a group of prime order $p$ such that
$|2A|<2.7652|A|$ and $|A|<1.25\cdot10^{-6}p$, then $A$ is contained in an
arithmetic progression with at most $|2A|-|A|+1$ terms, and $2A$ contains
an arithmetic progression with the same difference and at least $2|A|-1$
terms. This improves a number of previously known results.
\end{abstract}

\section{Introduction}

A classical result in additive combinatorics, Freiman's $(3n-4)$-theorem,
says that if $A$ is a finite set of integers satisfying $|2A|\le 3|A|-4$,
then $A$ is contained in an arithmetic progression of length $|2A|-|A|+1$.

It is believed that an analogue of Freiman's theorem holds for the
``not-too-large'' subsets of the prime-order groups; that is, if $\cA$ is a
subset of a group of prime order such that $|2\cA|\le 3|\cA|-4$ then, subject
to some mild density restrictions, $\cA$ is contained in an arithmetic
progression with at most $|2\cA|-|\cA|+1$ terms. The precise form of this
(and indeed, somewhat more general) conjecture can be found
in~\cite[Conjecture~19.2]{b:G13}.

For an integer $m\ge 1$, we denote by $\C_m$ the cyclic group of order $m$.
Let $p$ be a prime. Over sixty years ago, Freiman himself showed
\refb{Freiman61} that a subset $\cA\seq\C_p$ is contained in a progression
with at most $|2\cA|-|\cA|+1$ terms provided that $|2\cA|<2.4|\cA|-3$ and
$|\cA|<p/35$. Much work has been done to improve Freiman's result in various
directions; we list just a few results of this kind.

R{\o}dseth \refb{Rodseth06} showed that the assumption $|\cA|<p/35$ can be
relaxed to $|\cA|<p/10.7$. Green and Ruzsa \refb{GR06} pushed the doubling
constant from $2.4$ up to $3$, at the cost of a stronger density assumption
$|\cA|<p/10^{215}$. In \refb{SZ09}, Serra and Z{\'e}mor obtained a result
without any density assumption other than the conjectural one, but at the
cost of reducing essentially the doubling coefficient; namely, assuming that
$|2\cA|\le(2+\eps)|\cA|$ with $\eps<0.0001$. An improvement, allowing in
particular $\eps<0.1368$, was obtained by Candela, Gonz\'alez-S\'anchez, and
Grynkiewicz \refb{CGG21}. Candela, Serra, and Spiegel~\refb{CSS20} improved
the doubling coefficient to $2.48$ under the assumption $|\cA|<p/10^{10}$,
and this was further improved by Lev and Shkredov\refb{LS20} to $2.59$ and
$|\cA|<0.0045p$, respectively.

We have mentioned only several most relevant results; variations and
extensions, such as the results on the asymmetric sumset $\cA+\cB$ and
restricted sumset $\cA\dot+\cA$, are intentionally left out. A systematic
coverage of the topic can be found in \cite[Chapter~19]{b:G13}.

In this paper, we prove the following result.
\begin{theorem}\label{t:main}
Let $p$ be a prime, and suppose that a set $\cA\seq\C_p$ satisfies
$|2\cA|<2.7652|\cA|-3$. If $20\le|\cA|<1.25\cdot10^{-6}p$, then $\cA$ is
contained in an arithmetic progression with at most $|2\cA|-|\cA|+1$ terms,
and $2\cA$ contains an arithmetic progression with the same difference and at
least $2|\cA|-1$ terms.
\end{theorem}

Our argument follows closely that in~\refb{CSS20}. The improvements come
primarily from applying a result of Lev~\refb{Lev23} that establishes the
structure of small-doubling sets in cyclic groups (instead of an earlier
result of Deshouillers and Freiman~\refb{DF03}), and also from using an
estimate from a recent paper of Lev and Shkredov~\refb{LS20}.

In the next section we collect the results needed for the proof of
Theorem~\reft{main}. The proof itself is presented in the concluding
Section~\refs{proof}.

\section{Preparations}

This paper is intended for the reader familiar with the basic notions and
results from the area of additive combinatorics, such as the sumsets,
additive energy, Freiman isomorphism, Cauchy-Davenport and Vosper theorems,
the Pl\"{u}nnecke–Ruzsa inequality etc; they will be used without any further
explanations. Our notation and terminology are also quite standard. It may be
worth recalling, nevertheless, that a subset of an abelian group is called
\emph{rectifiable} if it is Freiman-isomorphic to a set of integers, and that
the \emph{additive dimension} of a subset $A\seq\Z$, denoted $\dim(A)$, is
the largest integer $d$ such that $A$ is Freiman-isomorphic to a subset of
$\Z^d$ not contained in a hyperplane. By $\phi_m$ we denote the canonical
homomorphism from $\Z$ onto the quotient group $\C_m\cong\Z/m\Z$. The
\emph{size} of an arithmetic progression is the number of its terms.


The core new component used in the proof of Theorem~\reft{main} is the
following result.
\begin{theorem}[Lev~{\cite[Theorem~1.1]{b:Lev23}}]\label{t:94}
Let $m$ be a positive integer. If a set $\cA\seq\C_m$ satisfies
$|2\cA|<\frac94|\cA|$, then one of the following holds:
\begin{itemize}
\item[(i)] There is a subgroup $\cH\le\C_m$ such that $\cA$ is contained in
    an $\cH$-coset and $|\cA|>C^{-1}|\cH|$, where
    $C=2\cdot10^5$. 
\item[(ii)] There is a proper subgroup $\cH<\C_m$ and an arithmetic
    progression $\cP$ of size $|\cP|>1$ such that $|\cP+\cH|=|\cP||\cH|$,
    $\cA\seq \cP+\cH$, and
       $$ (|\cP|-1)|\cH|\le|2\cA|-|\cA|. $$
\item[(iii)] There is a proper subgroup $\cH<\C_m$ such that $\cA$ meets
    exactly three $\cH$-cosets, the cosets are not in an arithmetic
    progression, and
      $$ 3|\cH|\le |2\cA|-|\cA|. $$
\end{itemize}
\end{theorem}

The following lemma originating from \refb{CSS20} relates the additive
dimension of a set with its rectifiability.
\begin{lemma}\label{l:higherdim}
Let $l$ be a positive integer, and suppose that $A$ is a set of integers
satisfying $\{0,l\}\seq A\seq[0,l]$ and $\gcd(A)=1$. If there is a proper
subgroup $H<\C_l$ such that the image of $A$ under the composite homomorphism
$\Z\to\C_l\to\C_l/H$  is rectifiable, then $\dim(A)\ge 2$.
\end{lemma}

Since the proof is just several lines long, we reproduce it here for the
convenience of the reader.
\begin{proof}
Writing $m:=l/|H|$, we identify the quotient group $\C_l/H$ with the group
$\C_m$, and the map $\Z\to\C_l\to\C_l/H$ with $\phi_m$. Let
$f\colon\phi_m(A)\to\Z$ be Freiman's isomorphism of $\phi_m(A)$ into the
integers. The set $\{(a,f(\phi_m(a)))\colon a\in A\}\seq\Z^2$ is easily seen
to be isomorphic to $A$, and to complete the proof we show that this set is
not contained in a line. Assuming the opposite, from
$f(\phi_m(0))=f(\phi_m(l))$ we derive that $f(\phi_m(a))$ attains the same
value for all $a\in A$. The same is then true for $\phi_m(a)$, showing that
$\phi_m(a)=\phi_m(0)=0$ for any $a\in A$; that is, all elements of $A$ are
divisible by $m$, contradicting the assumption $\gcd(A)=1$, except if $m=1$
in which case $H=\C_l$.
\end{proof}

From Theorem~\reft{94} and Lemma~\refl{higherdim} we deduce the key
proposition used in the proof of Theorem~\reft{main}.

\begin{proposition} \label{p:4k-8_partial}
Let $A$ be a finite set of integers satisfying
$|2A|<\frac{13}4\,|A|-\frac94$. If $\dim(A)=1$, then $A$ is contained in an
arithmetic progression with at most $2\cdot10^5|A|$ terms.
\end{proposition}

The proof essentially follows that of \cite[Proposition~2.3]{b:CSS20}, with
some simplifications, and with Theorem~\reft{94}
replacing~\cite[Theorem~1]{b:DF03}.
\begin{proof}[Proof of Proposition~\refp{4k-8_partial}]
Without loss of generality we assume that $\{0,l\}\seq A\seq[0,l]$ with an
integer $l>0$, and that $\gcd(A)=1$. We want to show that $l<2\cdot10^5|A|$.

Aiming at a contradiction, assume that $l\ge2\cdot10^5|A|$. Let
$\cA:=\phi_l(A)\seq\C_l$; thus, $|\cA|=|A|-1$. Since $\phi_l(a)=\phi_l(a+l)$
for any $a\in A\stm\{0,l\}$, and $\phi_l(0)=\phi_l(l)=\phi_l(2l)$, we have
$|2A|\ge|2\cA|+|A|$. It follows that
  $$ |2\cA| \le |2A|-|A| < \frac94\,|A| - \frac94 = \frac94\,|\cA|, $$
allowing us to apply Theorem~\reft{94}. We consider three possible cases
corresponding to the three cases in the conclusion of the theorem.

\medskip\noindent
Case (i): There is a subgroup $\cH\le\C_l$ such that $\cA$ is contained in an
$\cH$-coset and $|\cA|>C^{-1}|\cH|$, where $C=2\cdot10^5$. Since $\gcd(A)=1$,
the subgroup $\cH$ is not proper. Therefore
$|\cH|=l<2\cdot10^5|\cA|<2\cdot10^5|A|$, as wanted.

\medskip\noindent
Case (ii): There is a proper subgroup $\cH<\C_l$ and an arithmetic
progression $\cP\seq\C_l$ of size $|\cP|>1$ such that $|\cP+\cH|=|\cP||\cH|$,
$\cA\seq \cP+\cH$, and $(|\cP|-1)|\cH|\le|2\cA|-|\cA|$. The image of $\cA$
under the quotient map $\C_l\to\C_l/\cH$ is contained in an arithmetic
progression of size
  $$ |\cP| \le 1+(|2\cA|-|\cA|)/|\cH| \le 1+\frac54\,|\cA|/|\cH|
              < \frac54\,|A|/|\cH| < \frac12\,l/|\cH| = \frac12\,|\C_l/\cH|. $$
The difference of this progression is coprime with $|\C_l/\cH|$ in view of
the assumption $\gcd(A)=1$. Hence, the progression is rectifiable, and so is
the image of $A$ contained therein. The result now follows by applying
Lemma~\refl{higherdim}.

\medskip\noindent
Case (iii): There is a proper subgroup $\cH<\C_l$ such that $\cA$ meets
exactly three $\cH$-cosets, the cosets are not in an arithmetic progression,
and $3|\cH|\le |2\cA|-|\cA|$. In this case the image of $A$ in $\C_l/\cH$
consists of three elements not in an arithmetic progression; therefore the
image is isomorphic, say, to the set $\{0,1,3\}\seq\Z$, and an application of
Lemma~\refl{higherdim} completes the proof.
\end{proof}

\begin{lemma}[Freiman~{\cite[Lemma 1.14]{b:Freiman73}}]%
  \label{l:dimension_min}
For any finite, nonempty set $A$ of integers, writing $d:=\dim(A)$, we have
  $$ |2A| \ge (d+1)|A| - \binom{d+1}{2}. $$
\end{lemma}

\begin{lemma}[Candela-Serra-Spiegel~{\cite[Corollary~2.6]{b:CSS20}}]%
  \label{l:Freiman_2lines}
Let $A\seq\Z$ be a finite set with $\dim A=2$. If $|2A|\le\frac{10}{3}\,
|A|-7$, then $A$ is contained in the union of two arithmetic progressions,
$P_1$ and $P_2$, with the same difference, such that $|P_1\cup P_2|\le
|2A|-2|A|+3$ and the sumsets $2P_1$, $P_1+P_2$ and $2P_2$ are pairwise
disjoint.
\end{lemma}

The following result is, essentially, extracted from~\cite[Proof of
Theorem~3]{b:LS20}, with a little twist that will help us keep the remainder
terms under better control

For a prime $p$ and a subset $\cA\seq\C_p$, by $\hcA$ we denote the
non-normalized Fourier transform of the indicator function of $\cA$:
  $$ \hcA(\chi) = \sum_{a\in\cA} \chi(a);\quad \chi\in\hC. $$
The principal character is denoted by $1$. We let
  $$ \eta_\cA:=\max\{|\hcA(\chi)|/|\cA|\colon\chi\ne 1\}. $$

\begin{proposition}\label{p:rect}
Suppose that $p$ is a prime, and $\cA\seq\C_p$ is a nonempty subset of
density $\alp:=|\cA|/p<1/2$. If $|2\cA|=K|\cA|$ and $\cA$ is not an
arithmetic progression, then
  $$ (1-\alp K)(1-\eta_\cA^2)
             < 1 - K^{-1} - K^{-2} + (K - (1-2K^{-1})|\cA|)/|\cA|^2. $$
\end{proposition}

\begin{proof}
Let $\cS:=2\cA$ and $\cD:=\cA-\cA$. For a set $\cT\seq\C_p$ and element
$x\in\C_p$, we write $\cT_x:=\cT\cap(x+\cT)$; thus, $|\cT_x|$ is the number
of representations of $x$ as a difference of two elements of $\cT$, and in
particular $|\cT_0|=|\cT|$.

Consider the easily-verified identity
\begin{equation}\label{e:EVI}
  \frac1p\,\sum_{\chi\in\hC} |\hcA(\chi)|^2|\hcS(\chi)|^2
                                           = \sum_{x\in \cD} |\cA_x||\cS_x|.
\end{equation}
For the left-hand side using the Parseval identity we obtain the estimate
\begin{align}
  \frac1p\,\sum_{\chi\in\hC} |\hcA(\chi)|^2|\hcS(\chi)|^2
     &\le \frac1p\,|\cA|^2|\cS|^2 + \frac1p\,\eta_\cA^2|\cA|^2|\cS|(p-|\cS|) \notag \\
     &\le \alp K^2|\cA|^3 + \eta_\cA^2 K|\cA|^3(1-\alp K). \label{e:LHS}
\end{align}
To estimate the right-hand side we recall the \emph{Katz-Koester observation}
$\cA+\cA_x\seq\cS_x,\ x\in \C_p$. Let $N$ be the number of elements $x\in
\cD$ with $|\cA_x|=1$. Notice that $N\le|\cD|\le K^2|\cA|$; here the first
estimate is trivial, and the second is the Pl\"{u}nnecke–Ruzsa inequality.
From the assumption $\alp<1/2$ and the theorems of Cauchy-Davenport and
Vosper, we get
\begin{align}
   \sum_{x\in \cD} |\cA_x||\cS_x|
     &\ge \sum_{x\in \cD\stm\{0\}} |\cA_x||\cS_x| + |\cA||\cS|\notag \\
     &\ge \sum_{x\in \cD\stm\{0\}} |\cA_x||\cA+\cA_x| + |\cA||\cS|\notag \\
     &\ge \sum_{x\in \cD\stm\{0\}} |\cA_x|(|\cA|+|\cA_x|) - N + |\cA||\cS|  \notag \\
     &\ge \sum_{x\in \cD} |\cA_x|(|\cA|+|\cA_x|) - N +|\cA||\cS| - 2|\cA|^2  \notag \\
     &\ge |\cA|^3 + \E(\cA) - K^2|\cA| + (K-2)|\cA|^2 \label{e:RHS}
\end{align}
where $\E(\cA)=\sum_{x\in\cD}|\cA_x|^2$ is the additive energy of $\cA$, and
where the third estimate follows from Vosper's theorem if $|A+A_x|\le p-2$,
and otherwise from $|\cA+\cA_x|\ge p-1>2\alp p-1=2|\cA|-1\ge|\cA|+|\cA_x|-1$.

Combining~\refe{EVI}, \refe{LHS}, and \refe{RHS}, and using the basic bound
$\E(\cA)\ge|\cA|^3/K$, we get
  $$ \alp K^2|\cA|^3 + \eta_\cA^2 K|\cA|^3(1-\alp K) \ge (1+K^{-1})|\cA|^3
                                        - (K^2-(K-2)|\cA|)|\cA| $$
whence
\begin{gather*}
  \alp K + \eta_\cA^2 (1-\alp K)
                   \ge K^{-1}+K^{-2} - (K - (1-2K^{-1})|\cA|)/|\cA|^2, \\
  (\eta_\cA^2-1)(1-\alp K) \ge K^{-1}+K^{-2}-1
                                - (K - (1-2K^{-1})|\cA|)/|\cA|^2
\end{gather*}
which is equivalent to the inequality sought.
\end{proof}

\begin{corollary}\label{c:bias}
Let $\cA$, $\alp$, and $K$ be as in Proposition~\refp{rect}. If
$\alp<10^{-5}$, $K<2.7652$, and $|\cA|\ge10$, then
$\eta_\cA>\frac8{13}\,K-1$.
\end{corollary}

\begin{proof}
Assuming $\eta_\cA\le\frac8{13}\,K-1$ we get
  $$ 1-\eta_\cA^2\ge\frac{16}{13}\,K-\frac{64}{169}\,K^2
                             = \frac{16}{169}K(13-4K) $$
whence
\begin{equation}\label{e:KAalp}
   (1-\alp K)\frac{16}{169}K(13-4K)
                      < 1 - K^{-1} - K^{-2} + (K - (2-K^{-1})|\cA|)/|\cA|^2.
\end{equation}
The left-hand side is decreasing both as a function of $K$ and a function of
$\alp$, the right-hand side is an increasing function of $K$.
Therefore~\refe{KAalp} stays true with $K$ substituted by $2.7652$ and $\alp$
by $10^{-5}$; this results in a quadratic inequality in $|\cA|$ which is
false for $|\cA|\ge10$.
\end{proof}

The following lemma is standardly used to convert the ``Fourier bias''
(established in Corollary~\refc{bias}) into the ``combinatorial bias''.
\begin{lemma}[Freiman~\refb{Freiman73}]\label{l:arcs}
Suppose that $p$ is a prime, and $\cA\seq\C_p$ is a nonempty subset. There is
an arithmetic progression $\cP\sbs\C_p$ with $|\cP|\le (p+1)/2$ terms such
that
  $$ |\cA\cap\cP| > \frac12\,(1+\eta_\cA)|\cA|. $$
\end{lemma}

Finally, we need the symmetric case of a version of the $(3n-4)$-theorem due
to Grynkiewicz.
\begin{theorem}%
  [Special case of~{\cite[Theorem~7.1]{b:G13}}]\label{t:3n-4real}
Let $A$ be a finite set of integers. If $|2A|\le3|A|-4$, then $A$ is
contained in an arithmetic progression with at most $|2A|-|A|+1$ terms, and
$2A$ contains an arithmetic progression with the same difference and at least
$2|A|-1$ terms.
\end{theorem}

\section{Proof of Theorem~\reft{main}}\label{s:proof}

Throughout the proof, we identify $\C_p$ with the additive group of the
$p$-element field; accordingly, the automorphisms of $\C_p$ are identified
with the dilates. We write $d\ast\cA:=\{da\colon a\in\cA\}$ where $d$ is an
integer or an element of $\C_p$.

For $u\le v$, by $[u,v]$ we denote both the set of all integers $u\le z\le v$
and the image of this set in $\C_p$ under the homomorphism $\phi_p$. We may
also occasionally identify integers with their images under $\phi_p$. For
brevity, we write $p':=(p-1)/2$.

Assuming that $\cA\seq\C_p$ satisfies $|2\cA|\le K|\cA|-3$ with $K<2.7652$
and $20\le|\cA|<1.25\cdot10^{-6}p$, we prove that $\cA$ is contained in an
arithmetic progression with at most $(p+1)/2$ terms; equivalently, there is
an affine transformation that maps $\cA$ into a subset of an interval of
length at most $p'$. This will show that $\cA$ is rectifiable and imply the
result in view of Theorem~\reft{3n-4real}.

Let $\cA_0$ be a subset of $\cA$ of the largest possible size such that
$\cA_0$ is contained in an arithmetic progression with at most $(p+1)/2$
terms. We observe that, by the maximality of $|\cA_0|$, if $\cA_0\seq [0,l]$
with an integer $0\le l\le p'$, then the two intervals of length $p'-l-1$
adjacent to $[0,l]$ ``from the left'' and ``from the right'' do not contain
any elements of $\cA$; that is,
    $$ [l+p'+1,p-1]\cap\cA = [l+1,p']\cap\cA = \est. $$
Therefore
\begin{equation}\label{e:scorpio1}
  \cA \stm \cA_0 \seq [p'+1,p'+l] = p'+[1,l].
\end{equation}

Suppose first that $\cA_0$ is contained in an arithmetic progression with at
most $2\cdot 10^5|\cA_0|$ terms. Having applied a suitable affine
transformation, we assume that $\cA_0\seq[0,l]$ with $l<2\cdot10^5|\cA_0|$.
By \eqref{e:scorpio1}, we have
  $$ 2\ast\cA \seq (2\ast\cA_0) \cup [1,2l-1] \seq [0,2l]. $$
In view of $2l+1<4\cdot 10^5|\cA_0|\le p'$, this shows that the affine
transformation $z\mapsto 2z$ maps $\cA$ into an interval of length at most
$p'$, which is shown above to imply the result.

We therefore assume from now on that $\cA_0$ is not contained in an
arithmetic progression with $2\cdot 10^5|\cA_0|$ or fewer terms; in
particular, the set $\cA_0$ itself is not an arithmetic progression.

By Corollary~\refc{bias} and Lemma~\refl{arcs}, and in view of
$|\cA_0|\ge\frac12|\cA|\ge 10$ and
$|\cA_0|\le|\cA|<1.25\cdot10^{-6}p<10^{-5}p$, we have
\begin{equation}\label{e:413K}
  |\cA_0| > \frac{4}{13}K|\cA|,
\end{equation}
and it follows that
\begin{equation}\label{e:cA0cA0}
   |2\cA_0| \le |2\cA| \le K|\cA| - 3  < \frac{13}{4}|\cA_0| - \frac94.
\end{equation}
Recalling the way the set $\cA_0$ has been chosen, we find a set $A_0\seq\Z$
such that $\cA_0=\phi_p(A_0)$, $|A_0|=|\cA_0|$, and $A_0$ is contained in an
arithmetic progression with a most $p'+1$ terms; thus, $A_0$ is
Freiman-isomorphic to $\cA_0$, and as a result,
  $$ |2A_0| < \frac{13}{4}\,|A_0| - \frac94. $$
Since $\cA_0$ is not contained in an arithmetic progression with $2\cdot
10^5|\cA_0|$ or fewer terms, neither is $A_0$. (This does not follow from the
mere fact that $A_0$ and $\cA_0$ are Freiman-isomorphic, but does follow
immediately by observing that $\cA_0$ is the image of $A_0$ under a group
homomorphism.) Consequently, by Proposition~\refp{4k-8_partial}, we conclude
that $\dim(A_0)\ge 2$, and then, indeed, $\dim(A_0)=2$ by
Lemma~\refl{dimension_min}. Applying Lemma~\refl{Freiman_2lines}, we derive
that $A_0$ is contained in the union of two arithmetic progressions, say
$P_1$ and $P_2$, with the same difference, such that $|P_1\cup P_2|\le
|2A_0|-2|A_0|+3$ and the sumsets $2P_1$, $P_1+P_2$ and $2P_2$ are pairwise
disjoint. Hence, $\cA_0$ is contained in the union of the disjoint
progressions $\cP_1:=\phi_p(P_1)$ and $\cP_2:=\phi_p(P_2)$. Let
$\cA_1=\cA_0\cap \cP_1$ and $\cA_2=\cA_0\cap \cP_2$. Without loss of
generality, we assume that $|\cA_1|\ge|\cA_0|/2$.

Applying a suitable affine transformation,
we can arrange that
\begin{itemize}
\item[(i)] each of the progressions $\cP_1$ and $\cP_2$ has difference $1$
    or is a singleton;
\item[(ii)] there are integers $0\le b<c\le d$ such that $\cP_1\seq[0,b]$,
    $|\cP_1|=b+1$, and $\cP_2\seq[c,d]$, $|\cP_2|=d-c+1$;
\item[(iii)] the interval $[b,c]$ is at most as long as the interval
    $[d,p]$:
    \begin{equation}\label{e:0302b}
       c-b\le p-d.
    \end{equation}
\end{itemize}

Recalling~\refe{413K}, we obtain
\begin{multline*}\label{e:P1P2}
  b+d-c = |\cP_1|+|\cP_2|-2 \le |2\cA_0|-2|\cA_0|+1 \\
      \le |2\cA|-2|\cA_0|+1 < K|\cA|-\frac{8}{13}K|\cA|
                                        = \frac5{13}K|\cA| < 2|\cA|,
\end{multline*}
whence
\begin{equation}\label{e:0302a}
  b + (d-c) < 2|\cA|.
\end{equation}
Writing $n:=|\cA|$, we therefore have
\begin{equation}\label{e:bdc}
  \cA_1\seq[0,b]\seq[0,2n], \quad \cA_2\seq c+[0,d-c]\seq c+[0,2n],
\end{equation}
and also
  $$ (c-b)+(p-d) = p-(d-c)-b > p-2n. $$
Along with~\refe{0302b}, the last estimate gives $p-d \ge p'-n+1$ and,
consequently, $d\le p'+n$. In fact, we have
\begin{equation}\label{e:dul}
  4n < d < p'-4n;
\end{equation}
here the lower bound follows immediately from the assumption that $\cA_0$ is
not contained in a progression with $2\cdot10^5|\cA_0|$ or fewer terms, and
the upper bound follows by observing that if we had $p'-4n\le d\le p'+n$, in
view of~\refe{0302a} this would imply $[c,d]=[d-(d-c),d]\seq[d-2n,d]\seq
p'+[-6n,n]$ and, consequently,
$2\ast\cA_0\seq[0,2b]\cup[-12n-1,2n-1]\seq[-12n-1,4n]]$, also in a
contradiction with the same assumption.

We have $2\cA_0=2\cA_1\cup(\cA_1+\cA_2)\cup2\cA_2$ where the union is
disjoint; therefore, by the Cauchy-Davenport theorem,
 $$ |2\cA_0|\ge (2|\cA_1|-1) + (|\cA_1|+|\cA_2|-1) + (2|\cA_2|-1)
                                                      = 3|\cA_0|-3. $$
It follows that for any $a\in\cA\stm\cA_0$ we have
$(a+\cA_1)\cap(2\cA_0)\ne\est$, as assuming the opposite,
\begin{multline*}
 |2\cA| \ge |2\cA_0|+|a+\cA_1| \ge 3|\cA_0|-3+\frac12|\cA_0|
      > \frac72\cdot\frac4{13}K|\cA| - 3 = \frac{14}{13}K|\cA| - 3,
\end{multline*}
a contradiction. Therefore,
\begin{equation}\label{e:0402a}
  \cA\stm\cA_0 \seq 2\cA_0 - \cA_1 \seq \{0,c,2c\}+[-2n,4n].
\end{equation}
On the other hand, since $d<p'$, we can apply \refe{scorpio1} with $l=d$ to
get
\begin{equation}\label{e:0402b}
  \cA \stm \cA_0 \seq p'+[1,d].
\end{equation}
Comparing~\refe{0402a} and~\refe{0402b}, and observing that, in view
of~\refe{dul}, both intervals $[-2n,4n]$ and $c+[-2n,4n]$ are disjoint from
the interval $p'+[1,d]$, we conclude that
\begin{gather}
  \cA\stm\cA_0 \seq 2c+[-2n,4n] \label{e:2c2J1}
  \intertext{and, consequently,}
  \cA\seq\{0,c,2c\}+[-2n,4n]. \notag
\end{gather}

We notice that the set $2(\cA\stm\cA_0)$ is not disjoint from the set
$2\cA_0$ as otherwise we would get
\begin{multline*}
  |2\cA| \ge |2(\cA\stm\cA_0)| + |2\cA_0| \ge 2|\cA\stm\cA_0|-1 + 3|\cA_0|-3 \\
        = 2|\cA| + |\cA_0| - 4 \ge \left(2+\frac4{13}K\right)|\cA|-4 > K|\cA|-3.
\end{multline*}
Since $2(\cA\stm\cA_0)\seq 4c+[-4n,8n]$ by~\refe{2c2J1}, and
$2\cA_0\seq\{0,c,2c\}+[0,4n]$ in view of~\refe{bdc}, we conclude that
$kc\in[-8n,8n]$ for some $k\in\{2,3,4\}$. Therefore
$k\ast\cA_0\seq\{0,kc\}+[0,2kn]\seq [-8n,(8+2k)n]$. Hence, $\cA_0$ is
contained in an arithmetic progression with at most
$(16+2k)n+1<25n<2\cdot10^5|\cA_0|$ terms, a contradiction.

\bigskip

\end{document}